\documentclass[12pt]{amsart}
\usepackage[utf8]{inputenc}
\usepackage{amssymb}
\usepackage{geometry}
\usepackage{amsmath,amscd}
\usepackage{graphicx}
\usepackage{color,hyperref}
\usepackage{tikz}
\usepackage{tikz-cd}
\usepackage{amsfonts}
\usetikzlibrary{matrix,arrows}
\usepackage{mathrsfs}
\usepackage{amssymb}
\usepackage{adjustbox}
\usepackage{physics}
\begin{document}
	
	\newcommand{\mf}{\mathfrak}
	\newcommand{\mc}{\mathcal}
	\newcommand{\mb}{\mathbf}
	
	\newcommand{\R}{\mathbf R}
	\newcommand{\C}{\mathbf C}
	\newcommand{\Q}{\mathbf Q}
	\newcommand{\Z}{\mathbf Z}
	\newcommand{\F}{\mathbf F}
	\newcommand{\N}{\mathbf N}

	\newcommand{\Fix}{\textnormal{Fix}}
	\newcommand{\End}{\textnormal{End}}
	\newcommand{\Frob}{\textnormal{Frob}}

	\newcommand{\sm}[4]{{
			\left(\begin{smallmatrix}{#1}&{#2}\\{#3}&{#4}\end{smallmatrix}\right)}}
	\newcommand{\sv}[2]{
		\genfrac(){0pt}{1}{#1}{#2}}

	\numberwithin{equation}{section}
	
	\theoremstyle{plain}
	\newtheorem{theorem}{Theorem}
	\newtheorem{lemma}[theorem]{Lemma}
	\newtheorem{proposition}[theorem]{Proposition}
	\newtheorem{corollary}[theorem]{Corollary}
	\newtheorem*{t2}{Theorem 2}
	\newtheorem*{t4}{Theorem 4}
	
	\theoremstyle{definition}
	\newtheorem{definition}[theorem]{Definition}
	\newtheorem{example}[theorem]{Example}
	\newtheorem{conjecture}[theorem]{Conjecture}
	\newtheorem{remark}[theorem]{Remark}

	\title{$\alpha$-Hölder Integrability Exponent }
	\author{SADIK TERZ\.{I}}
	\address{Department  of Mathematics,  Middle East  Technical University,  06800 Ankara,  Turkey }
\email{sterzi@metu.edu.tr}

\author{Ozcan Yazici}
\address{Department  of Mathematics,  Middle East  Technical University,  06800 Ankara,  Turkey }
\email{oyazici@metu.edu.tr}
    
	
	\subjclass[2010]{}
	
	
	\maketitle
	
	\begin{abstract} For a plurisubharmonic function $\varphi$ on $\Omega\subset \mathbb C^n$ we  defined $\alpha$-Hölder integrability exponent $c_{\alpha}(\varphi)$ and  obtained a lower bound in terms of Lelong number $\nu(\varphi)$  and intersection numbers $e_j(\varphi)$. This lower bound improves the earlier result of Kaufmann \cite{Kf} on the integrability of plurisubharmonic functions with respect to Monge-Amp\'ere measure with $\alpha$-Hölder continuous potentials. We also give a sharp lower bound for $\alpha$-Hölder integrability exponent  of certain type of plurisubharmonic functions.\\
		
 
	\end{abstract}

	\vskip 0.5cm
	\section{Introduction}
Let $\Omega$ be an open subset of $\mathbb{C}^n$ and $\varphi$ be a plurisubharmonic (psh) function on $\Omega.$ The $\textit{Lelong number}$ of $\varphi$ at $a\in \Omega$ can be defined as 
$$ \nu(\varphi,a): = \displaystyle \liminf_{z\rightarrow a}\dfrac{\varphi(z)}{\log\mid\mid z-a \mid\mid }. $$ Equivalently Lelong number of $\varphi$ at $a$ can be defined by $$\nu(\varphi,a) = \sup\{ \gamma: \varphi(z)\leq \gamma\log\mid\mid z-a\mid\mid  +   \text{ }O(1) \text{ as } z\rightarrow a \}.$$ The deep theorem of Siu \cite{Si} states that the function $z\mapsto \nu(\varphi,z)$ is upper semicontinuous with respect to the Zariski topology, that is, for every $c>0$, the set $\{ a\in \Omega : \nu(\varphi,a) \geq c\}$ is a closed analytic subvariety of $\Omega$. We also note that it is upper semicontinuous with respect to the usual Euclidean topology.
\par The operator $(dd^c- )^n$, acting on locally bounded psh functions, is called the $\textit{complex Monge-}$ $\textit{Amp\'ere operator}$ according to definition of Bedford-Taylor \cite{BT1}, \cite{BT2}. This definition can also be extended to psh functions with isolated or compactly supported poles by \cite{D1}. Let $PSH(\Omega)$ (resp. $PSH^-(\Omega)$) be the set of psh (resp. negative psh) functions on $\Omega$. Following Cegrell \cite{Cg}, one introduces certain classes $\mathcal{E}_0(\Omega), \mathcal{F}(\Omega) \text{ and }\mathcal{E}(\Omega)$ of negative psh functions on a bounded hyperconvex domain $\Omega$, and it is proved in \cite{Cg} that the class $\mathcal{E}(\Omega)$ is the biggest subset of $PSH^-(\Omega)$ on which the Monge-Amp\'ere operator is well defined. If one considers a general complex manifold $X$ and drops the negativity assumption on the psh functions involved, one can in fact extend the Monge-Amp\'ere operator to the class $$ \tilde{\mathcal{E}}(X)\subset PSH(X)$$ of psh functions which, on a neighborhood $\Omega$ of $x_0\in X$, are equal to a sum $u+v$ with $u\in \mathcal{E}(\Omega)$ and $v\in C^{\infty}(\Omega)$; this is the biggest subclass of functions of $PSH(X)$ on which the Monge-Amp\'ere operator is locally well defined.
\par In \cite{DK}, Demailly and Koll\'ar introduced the $\textit{integrability exponent}$ of $\varphi$ at $a$ as $$ c(\varphi,a)= \sup\{ c>0:e^{-2c\varphi} \text{ is } L^1 \text{ on a neighborhood of } a\}.$$ The integrability exponent $c(\varphi,a)$ is a another way of measuring the singularity of a psh function $\varphi$ at a point $a\in \Omega$. It was used  as a tool to study several types of algebraic and analytic objects, such as holomorphic functions, divisors, coherent ideal sheaves, positive closed currents in  \cite{DK}.

In \cite{DP}, for  $\varphi \in \tilde{\mathcal{E}}(X)$ and $a\in \Omega$,  Demailly and Ph\d{a}m defined  $\textit{intersection numbers}$ of $\varphi$ at $a$ as follows: 
$$ e_j(\varphi,a) = \int_{\{a\}}^{}(dd^c\varphi)^j\wedge (dd^c\log \mid\mid z-a \mid \mid)^{n-j},$$ which can be seen also as the Lelong number of $(dd^c\varphi)^j$ at $a$ for $j=0,1, \cdots ,n$.  We notice that the Lelong number $\nu(\varphi,a) $ is equal to the first intersection number $e_1(\varphi,a)$ of $\varphi$ at $a\in \Omega$  and $e_0(\varphi,a)=1.$ We also set $$\Sigma(\varphi,a):= \displaystyle\sum_{j=0}^{n-1}\dfrac{e_j(\varphi,a)}{e_{j+1}(\varphi,a)} $$ for $\varphi \in \tilde{\mathcal{E}}(\Omega)$ and $a\in \Omega$.
\par A classical theorem of Skoda \cite{Sk} states that if $\nu(\varphi,a)<2$, then $e^{-\varphi}$ is integrable in a neighborhood of a with respect to the Lebesgue measure. This theorem, on the one hand, translates as $c(\varphi,a)\geq \nu(\varphi,a)^{-1}$, and on the other hand, translates as $c(\varphi,a)\leq n \nu(\varphi,a)^{-1}$. In \cite{DP}, the authors find sharp lower bound for $c(\varphi,a)$ as $$ c(\varphi,a) \geq \Sigma(\varphi,a)$$ for a psh $\varphi$ and $a \in \Omega$. 
\par One can put forward similar questions for "integrability like exponents" when the Lebesgue measure is replaced by other type of measures. In this direction, one introduces a general Monge-Amp\'ere mass with a Hölder-continuous potential in below. This class of measures is important in many applications of pluripotential theory to complex geometry and dynamics and has been the topic of recent
research. The reader may consult \cite{DN} and \cite{Kl} for a complete characterization of these
measures and  \cite{DDG} for some of its properties. 

\par We say that a psh function $u:\Omega \longrightarrow [-\infty,\infty)$ is a $\alpha-\textit{Hölder continuous psh}$ function on $\Omega$ if for all $z,z' \in \Omega$, we have the following inequality $$ \mid u(z)-u(z')  \mid \leq C\mid \mid z-z'\mid \mid^{\alpha}$$ for some constant $C$ and for some exponent $0<\alpha \leq 1$. The measure $\mu = (dd^cu)^n$ for $\alpha$-Hölder psh function $u $ with Hölder exponent $0<\alpha\leq 1$ is called a $\textit{Monge-Amp\'ere measure with Hölder}$ $\textit{continuous potential } u$. 
\begin{definition}
    A measure $\mu$ on a complex manifold $X$ is called $\textit{locally moderate}$ if for any open set $U\subset X$, any compact set $K\subset U$, and any compact family $\mathcal{F}$ of psh functions on $U$, there are constants $\beta,C>0$ such that $$ \int_{K}^{}e^{-\beta\psi}d\mu \leq C$$ for every $\psi \in \mathcal{F}$.
\end{definition}
    We note that the measures of the form $\mu = (dd^cu)^n$ for an $\alpha$ -Hölder continuous psh $u$ are always locally moderate measures \cite{DNS}. We introduce the $\alpha-\textit{Hölder integrability exponent } c_{\alpha}(\varphi,a) $ of $\varphi$ at $a\in \Omega$  as follows 
    \[
   c_{\alpha}(\varphi,a) =\sup \left \{ c\geq 0 \ 
 \int_{K}^{}e^{-2c\varphi}(dd^cu)^n < \infty, \; \text{ for any}\; \alpha-\text{Hölder cont. psh } u \text{ on } \Omega\right \}.
\] 
Here $K$ is a compact neighborhood of $a$ in $\Omega$. Theorem 4.1 in \cite{Kf} implies that   $$ c_{\alpha}(\varphi,a)\geq \dfrac{\alpha}{\nu(\varphi, a)(\alpha + n(2-\alpha))}.$$

 Using the techniques that appeared in \cite{Kf} and \cite{DNS} and the main result in \cite{DP} we improved the lower bound for  $ c_{\alpha}(\varphi,a)$.


\begin{theorem}\label{main} Let $\varphi \in \tilde {\mathcal {E}}.$ Then  $$ c_{\alpha}(\varphi,a) \geq \dfrac{\alpha}{\alpha + n(2-\alpha)}\cdot \Sigma(\varphi,a).$$  

\end{theorem}
 
 In order to obtain even better lower bound  for $\alpha$-Hölder integrability exponents we want to use directly complex Monge-Amp\'ere masses of level sets $\{ \varphi < -N \}$ instead of Lebesgue masses of that sets.  In this direction, we will use partial refined Lelong numbers defined in \cite[Section 7]{Ki} and restrict ourselves to a special type of psh functions, i.e., psh functions bounded below by a toric psh function, so that we may majorize the level sets $\{\varphi <-N\}$ by balls $B(0,r)$ where $r$ depends on the natural number $N$. 

We call a psh function toric if $\varphi(z_1,...,z_n)=\varphi(|z_1|,\dots, |z_n|)$ depends only on $|z_j|$ for all $j.$ 
 
\par In \cite[Section 7]{Ki}, the author defines $\textit{partial refined Lelong number}$ $\nu_{f,k}$ for a psh function $f$ in $\Omega$ as follows:
$$ \nu_{f,k}(x,y_1,\cdots,y_k)= \displaystyle\lim_{t\rightarrow -\infty} \dfrac{ V_{f,k}(x,ty_1,\cdots ,ty_k) }{t}  $$
where $ x\in \Omega, (y_1,\cdots,y_k) \in \mathbb{R}_+^k$ and  $$V_{f,k}(x,y_1,\cdots ,y_k) =  \displaystyle \sup\{f(x_1 + z_1, \cdots , x_k+ z_k, x_{k+1},\cdots, x_n): |z_j|=e^{y_j} \}.$$  
\par In the proof of Proposition \ref{toric bdd}, we need a small modification of Kiselman partial refined lelong numbers and hence we define \textit{partial like refined Lelong numbers} in the following setup: \\
Let $f$ be a psh function in $\Omega \subset \mathbb{C}^n$ and $J=\{j_1,\cdots, j_k\}$ be any subset of $\{1,2,\cdots , n\}$. We define a new psh function $f_J(z)$ from the old one by setting 
$$  f_J(z_{j_1},\cdots, z_{j_k})= f(z_1,\cdots, z_n)|_{\Omega_J}$$ 
where $\Omega_J =\Omega \cap \{z_l=x_l: l\notin J \}$.

\begin{definition}
The number $$ \nu_{f,J}(x,y_J) = \displaystyle\lim_{t\rightarrow -\infty} \dfrac{V_{f,J}(x,ty_J)}{t}$$ is called partial like refined Lelong number of $f$ where $x\in \Omega$, $y_J=(y_{j_1},\cdots,y_{j_k})$ and $V_{f,J}(x,y_J)= \text{sup}\{ f_J( x_{j_1} +z_{j_1},\cdots, x_{j_k} + z_{j_k}) : |z_{j_l}|=e^{y_{j_l}} \text{ for } l=1,\cdots,k\}$.
\end{definition}

For simplicity, we will denote $\nu_{f,J}(x,y_J)$ by $\nu_{f,j}(x,y_j)$ when $J=\{j\}$ is the one point set. By using this type of Lelong numbers, we obtain the following integrability result for the special family of psh functions:

\begin{theorem} \label{main2}
     Let $f$ be a toric psh function in the domain $\Omega \subset \mathbb{C}^n$ with an isolated singularity at $0\in \Omega$. Let $\varphi$ be a psh function on $\Omega$ such that $\varphi(z) \geq  f(z)$ on a compact neighborhood $K$ of $0$ and $u$ be an $\alpha$-Hölder psh function for some $\alpha$ with $0<\alpha \leq 1$. Suppose that $\nu_{f,j}(0,1)>0$ for all $j:1\dots n.$ If we set $$\nu = \displaystyle\max_j(\nu_{f,j}(0,1))$$ and assume that $\nu < \alpha n$, then the integral  
$$   \displaystyle\int\limits_{K}^{}e^{-\varphi}(dd^cu)^n $$ is finite.

\end{theorem}
The following example shows that the condition on $\nu$ in Theorem \ref{main2} is sharp.  

\begin{example}\label{ex5} Let $0<\alpha<1$ and $u(z)=|z_1|^{\alpha}+\dots +|z_n|^{\alpha}$ and $\varphi(z)=c\log||z||.$ Then $(dd^c u)^n=n!\left( \frac{\alpha}{2} \right)^{2n}|z_1|^{\alpha-2}\dots |z_n|^{\alpha-2}\lambda_{2n}$ and $$\displaystyle I=\int_{B(0,\epsilon)}e^{-\varphi}(dd^c u)^n= \int_{B(0,\epsilon)}\frac{|z_1|^{\alpha-2}\cdots|z_n|^{\alpha-2}}{(|z_1|^2+\cdots+|z_n|^2)^{c/2}}\,dz_1\cdots dz_n .$$ 
Using polar coordinates with $r=\Big(\sum_{i=1}^n|z_i|^2\Big)^{1/2},\; z_i=r\,\omega_i$, we get that $I\approx\int_0^{\epsilon} r^{\,n\alpha-c-1}\,dr.$
Thus the integral $I$ converges if and only if $\nu=c<n\alpha.$
\end{example}

\par The last part of this manuscript is devoted to find relaxed integrability conditions of complex Monge-Amp\'ere measures for a specific class of Hölder continuous psh functions. We will consider the Dirichlet problem  
\[ \text{Dir}(\Omega, \psi,gd\mu): \begin{cases} 
      u \in \text{psh}(\Omega)\cap C(\bar{\Omega}) & , \\
      (dd^cu)^n=gd\mu \text{ in } \Omega &, \\
      u=\psi \text{ on } \partial\Omega
   \end{cases} \]
where $\mu$ is finite Borel measure on a bounded domain 
$\Omega \subset \mathbb{C}^n$ and $g \in L^p(\Omega, \mu)$ for $p>1$. When the boundary function $\psi$ belongs to $C^{0,\alpha}(\partial\Omega)$, the unique solution to complex Monge-Amp\'ere equation becomes $\alpha'$-Hölder continuous psh function $u$ in  $\bar{\Omega}$. In this direction, we restrict ourselves to cases where $\mu \in \mathcal{F}^1_{2n-2 + \epsilon}(B)$  is the Hausdorff-Riesz measure of order $2n-2+\epsilon$ \cite[Section 2]{Z} for $0<\epsilon \leq 2$ and produce the result Theorem \ref{ex}.

  \section{Proof of Theorem \ref{main}  } 

We observe that  for any psh   $\varphi$ satisfying $\displaystyle\frac{1}{\Sigma(\varphi, a)}<2$ the lower bound  $\Sigma(\varphi,a)\leq c(\varphi,a)$ in \cite{DP} implies that  $\displaystyle\int_{\{a\}}^{}e^{-\varphi}<\infty.$
    \begin{lemma}{\label{l2}}
        Let $\varphi$ be a psh function in a domain $\Omega \subset \mathbb{C}^n$ and $a\in\Omega$. Then for every $\gamma<2\Sigma(\varphi,a)$, there is a constant 
        $C_{\gamma}=C_{\gamma}(\varphi, \Omega, a)$ and a compact set $K$ near $a$ such that 
        $$ \lambda_{2n}(K\cap \{\varphi\leq M \})\leq C_{\gamma}e^{-\gamma M}$$
        where $\lambda_{2n} $ denotes the Lebesgue measure in $\mathbb{C}^n$.
    \end{lemma}
    \begin{proof}
By \cite{DP}, $c(\varphi,a)\geq \Sigma(\varphi,a)>\frac{\gamma}{2}$. Thus there is a compact neighborhood $K$ of $a$ such that  $\displaystyle\int_{K}^{}e^{-\gamma\varphi}<\infty$. We now set $ C_{\gamma}=C_{\gamma}(\varphi, \Omega , a)=  \displaystyle\int_{K}^{}e^{-\gamma\varphi}< \infty$. On the set $K\cap \{\varphi \leq -M\}$ $1\leq e^{\gamma(-\varphi-M)}$. We obtain that 
$$ \lambda_{2n}(K\cap \{\varphi \leq -M\}) \leq \displaystyle\int_{K}^{}1 \leq \displaystyle\int_{K}^{}e^{\gamma(-\varphi-M)}=C_{\gamma}e^{-\gamma M}.$$

    \end{proof}

Now we prove our first  main result which implies Theorem \ref{main}.    
    \begin{theorem}\label{t4}
        Let $0< \alpha \leq 1$, $u$ be an $\alpha$-Hölder continuous psh function in $\Omega \subset \mathbb{C}^n$,  $\varphi\in \tilde{\mathcal{E}}(\Omega)$ and $a\in \Omega$. If  $\dfrac{1}{\Sigma(\varphi,a)} < \dfrac{2\alpha}{\alpha + n(2-\alpha)}$ then there is a neighborhood $K\subset \Omega$ of $a$ such that
        $$ \displaystyle\int_{K}^{}e^{-\varphi}(dd^cu)^n< \infty.$$ 
        \end{theorem}
    \begin{proof}

We assume that $a=0$ and $\varphi$ is negative. We also set $\omega = dd^c\mid \mid z \mid \mid ^2$. For $N>0$, we define $\varphi_N= \text{ max}\{\varphi, -N\}$ and $\psi_N= \varphi_{N-1}-\varphi_N$. We note that $0\leq \psi_N \leq 1$, $\psi_N$ is supported in $\{\varphi < -N+1\}$, and $\psi_N \equiv 1$ in $ \{\varphi < -N \}$. We have that 

\begin{align}  \label{(1.1)}
  \displaystyle\int\limits_{}^{}e^{-\varphi}(dd^cu)^n &= \displaystyle\sum_{N=0}^{\infty}\int\limits_{\{ -N\leq \varphi < -N+1\}}^{}e^{-\varphi}(dd^cu)^n \nonumber \\ & \leq\displaystyle\sum_{N=0}^{\infty} e^N \int\limits_{\{ -N\leq \varphi < -N+1\}}^{}(dd^cu)^n \nonumber \\ & \leq  \displaystyle\sum_{N=0}^{\infty} e^N \int\limits_{\{ -N\leq \varphi < -N+1\}}^{}\psi_{N-1}(dd^cu)^n.
  \end{align} 

By the assumption that $\dfrac{1}{\Sigma(\varphi,0)} < \dfrac{2\alpha}{\alpha + n(2-\alpha)}$ we have $\dfrac{1}{\Sigma(\varphi,0)} < \dfrac{2\alpha}{\alpha + n(2-\alpha)} - 2\sigma$ for some $\sigma >0$. Hence

    $$ \dfrac{\alpha +n(2-\alpha)}{\alpha} + \delta < 2\Sigma(\varphi,0)$$ for some $\delta>0.$
    Therefore, by Lemma \ref{l2}, in a small neighborhood  $K$ of $0$ 
    
    \begin{eqnarray} \label{volume}    
   \lambda_{2n}(K\cap \{ \varphi <-N+1\}) \lesssim e^{-\left ( \dfrac{\alpha + n(2-\alpha)}{\alpha} + \delta \right) N} = e^{-( 1 + \delta )N} \cdot e^{-\left ( \dfrac{ n(2-\alpha)}{\alpha}\right) N}.  
   \end{eqnarray}

    The symbol $\lesssim$ means that left hand side is smaller than or equal to a constant time the right hand side, constant being independent from $N$.
    \par One may suppose that $u$ is defined on $ B_{2r} \subset K$ by taking smaller $r$ if necessary. By subtructing a constant from $u$, we may assume that $u\leq -1$. We will estimate the integral $$\int\limits_{\{ -N\leq \varphi < -N+1\}}^{}\psi_{N-1}(dd^cu)^n.$$
    Fix a smooth cut-off function $\chi$ with compact support in $B_{r-\rho} $, $0\leq \chi\leq 1$ and $\chi\equiv 1$ on $B_{r-2\rho}$ where $\rho<\frac{r}{4}$. By integration by part formula (see \cite{DNS}), we can split the integral by 
    
    \begin{eqnarray}\int\limits_{B_r}^{}\chi\psi_{N-1}(dd^cu)^n&=&-\int\limits_{B_r\setminus B_{r-3\rho}}^{}dd^c\chi\wedge \psi_{N-1}u(dd^cu)^{n-1}-\int\limits_{B_r\setminus B_{r-3\rho}}^{}d\chi\wedge \psi_{N-1}d^cu\wedge (dd^cu)^{n-1}\nonumber\\ &+& \int\limits_{B_r\setminus B_{r-3\rho}}^{}d^c\chi\wedge \psi_{N-1}du\wedge (dd^cu)^{n-1}+ \int\limits_{B_{r-\rho}}^{}\chi u dd^c\psi_{N-1}(dd^cu)^{n-1}.
    \end{eqnarray}
    Since $u$ is smooth in $B_r\setminus B_{r-3\rho}$ and supp $\psi_{N-1}\subset \{\varphi\leq -N+1\}$, it follows from \ref{volume} that the first three integrals on the right hand side is bounded above by  $e^{-( 1 + \delta )N} \cdot e^{-\left ( \dfrac{ n(2-\alpha)}{\alpha}\right) N}$.
The last integral can be written as $$\int\limits_{B_{r-\rho}}^{}\chi u_{\epsilon} dd^c\psi_{N-1}\wedge(dd^cu)^{n-1}+\int\limits_{B_{r-\rho}}^{}\chi (u-u_{\epsilon}) dd^c\psi_{N-1}\wedge(dd^cu)^{n-1}$$
where $u_{\epsilon}$ is the smooth regularization of $u$ satisfying $$||u-u_{\epsilon}||_{\infty}\leq \epsilon^{\alpha}, \; ||u_{\epsilon}||_{\mathcal C^2}\leq \epsilon^{\alpha-2}, \; \epsilon=e^{-(1/\alpha+c)N}, \; 0<c<\frac{\delta}{n(2-\alpha)}.$$ 
Since $(dd^cu)^{n-1}$ is locally moderate, second integral is bounded by $c||u-u_{\epsilon}||\lesssim e^{-(1+c\alpha)N}.$

To deal with the first integral, we use integration by part successively to move $dd^c$ from $u$ to $u_{\epsilon}$. All resulting integrals, except the last one  are bounded by $Ce^{-( 1 + \delta )N} \cdot e^{-\left ( \dfrac{ n(2-\alpha)}{\alpha}\right) N}$. The last one is estimated by $$\int\limits_{B_r\setminus \rho}^{}\chi\psi_{N-1}(dd^cu_{\epsilon})^n\lesssim (||u_{\epsilon}||_{\mathcal C^2})^n\lambda_{2n}(B_r\cap \{\varphi \leq -N+1\})\lesssim e^{(cn(2-\alpha)-(1+\delta))N}.$$ 
We now combine all these estimates in \ref{(1.1)}:

$$\displaystyle\int\limits_{}^{}e^{-\varphi}(dd^cu)^n\lesssim\sum_{n=0}^{\infty}e^{-\delta N-(n(2-\alpha)/\alpha)N}+e^{-c\alpha N}+e^{(cn(2-\alpha)-\delta)N}<\infty.$$
\end{proof}
As a corollary, we obtain the lower bound for $c_{\alpha}(\varphi,z)$ in Theorem \ref{main}.

\begin{corollary} \label{in1}
Let $\varphi$ be a psh function on a domain $\Omega \subset \mathbb{C}^n$ and $a\in \Omega$. Then we have  $$ c_{\alpha}(\varphi,a) \geq \dfrac{\alpha}{\alpha + n(2-\alpha)}\cdot \Sigma(\varphi,a)$$  
\end{corollary}
\begin{proof}

We note that $\Sigma(\varphi,a) = \gamma \Sigma(\gamma\varphi,a)$ for all $\gamma>0$ and for all $a\in \Omega$. Let $\gamma < \dfrac{\alpha}{\alpha + n(2-\alpha)}\Sigma(\varphi, a).$ Then $$\frac{1}{\Sigma(2\gamma\varphi,a)}=\frac{2\gamma}{\Sigma(\varphi,a)}<\frac{2\alpha}{\alpha+n(2-\alpha)}.$$ By Theorem \ref{t4} $     \displaystyle\int\limits_{K}^{} e^{-2\gamma\varphi}(dd^cu)^n < \infty$ for some neighborhood $K$ of $a$.  Thus $$ c_{\alpha}(\varphi,a) \geq \dfrac{\alpha}{\alpha + n(2-\alpha)}\cdot \Sigma(\varphi,a).$$
\end{proof}

\begin{remark} By Proposition 4.4 in \cite{Kf}, there exists a Monge-Amp\'ere  mass $\mu$ with $\alpha$-H\"{o}lder continuous psh potential such that $\int_K e^{-\varphi}d\mu=\infty$ for any psh $\varphi$ with $\nu(\varphi,0)>n\alpha$ where $K$ is any neighborhood of $0$. This implies that $$c_{\alpha}(\varphi,0) \leq \dfrac{n\alpha}{2\nu(\varphi,0)}.$$
Combining with Corollary \ref{in1} we obtain that $$\dfrac{\alpha}{\alpha + n(2-\alpha)}\cdot \Sigma(\varphi,a) \leq c_{\alpha}(\varphi,a) \leq \dfrac{n\alpha}{2\nu(\varphi,a)}$$  
for $a\in \Omega$.
\end{remark}
\section{Monge-Amp\'ere masses on level sets}
We will first estimate complex Monge-Amp\'ere masses of balls $B(0,r)$ then estimate that of the level sets $\{\varphi < -N\}$ for all $N\in \mathbb{N}$ and for all psh function $\varphi$ bounded below by some toric psh function. Therefore, to achieve our aim, we will focus initially on the following result:
\begin{proposition} \label{toric bdd}
 Let $f$, $\varphi$, $\nu$ and $K$ be as in Theorem \ref{main2}. Then there exists a constant $C>0$ such that $\varphi(z)\geq  \nu \max\{\log||z_j||:j=1,\cdots,n\} -C$  on $K$ where $\nu = \max\{\nu_{f,j}(0,1):\; \text{ for }j=1,\cdots,n\}   $. 
\end{proposition}
\begin{proof}
   Consider $z=(z_1,\cdots, z_n)\in \Omega$, $y=(y_1,\cdots, y_n ) \in \mathbb{R}^n_+$ and one point subset $J=\{ j\} \subset \{1,\cdots,n\}$ for $j=1,\cdots,n$. The functions $f_j(z_j):=f|_{\Omega_j}$ are psh on the set $$\Omega_j:=\Omega\cap \{z_i=0, \; \text{for}\; i\neq j \}.$$ By  our assumption $f_j$ has isolated singularity at $0\in \Omega_j$. As $f_j$ is upper semicontinuous on $\Omega$ and has isolated singularity at $0\in \Omega_j$, we may assume that $f_j(z_j)\leq 0$ near $0$. Then the partial like refined Lelong numbers $\nu_{f,j}(0,y_j)$ become 
   $$ \nu_{f,j}(0,y_j) = \displaystyle\lim_{t\rightarrow -\infty} \dfrac{ V_{f,j}(0,ty_j) }{t}=  \displaystyle\lim_{t\rightarrow -\infty} \dfrac{ \sup\{ f_j(z_j) : |z_j|=e^{ty_j}\} }{t}.     $$ 
   As $f_j(z_j)$ is toric psh function on $\Omega_j$, we have 
   $$ \nu_{f,j}(0,y_j) =  \displaystyle\lim_{t\rightarrow -\infty} \dfrac{ f_j(e^{ty_j}) }{t}.  $$
   If we set $t= \dfrac{\text{log}(||z_j||)}{y_j}$, then we obtain the following equality
   $$  0<\dfrac{ \nu_{f,j}(0,y_j) }{y_j}  = \displaystyle\lim_{||z_j||\rightarrow 0} \dfrac{ f_j(z_j) }{\log(||z_j||)}. $$ 
   Note that  $f_j$ has isolated singularity at $0$  and the above limit is equal to the Lelong number $\nu(f_j,0)$ of $f_j$ at $0$. Thus the ratio $\dfrac{ \nu_{f,j}(0,y_j) }{y_j}$ is equal to $\nu(f_j,0)$ and hence independent of $y_j$.  Since $\dfrac{ f_j(z_j) }{\text{log}(||z_j||)} $ decreases to the Lelong number $\nu(f_j,0)=\nu_{f,j}(0,1)$ for $j=1,\cdots,n$ there exists positive constant $C_j$ such that $$f_j(z_j) \geq \nu_{f,j}(0,1)\text{log}(||z_j||) - C_j$$ for $j=1,\cdots,n$, $||z_j||\leq r_j$ and $0<r_j<<1$. As $f(z)$ is a toric psh function on $\Omega$, $f(z)$ is increasing in each variable on $\mathbb{R}^n_+$. Therefore, $\forall z=(z_1,\cdots,z_n) \in K$ $$ \varphi \geq f(z)\geq f_j(z_j) \geq \nu_{f,j}(0,1)\text{log}(||z_j||)-C_j \geq \nu \text{log}(||z_j||)-C $$ where $C=\text{max}_j(C_j)$ and $\nu=\text{max}_j(\nu_{f,j}(0,1))$.  
\end{proof}
\par We note that there are plenty of examples of psh functions which satisfy the hypothesis of Proposition \ref{toric bdd}: 
\begin{example}\cite[ Example 2.4]{DK}
  Let $\varphi(z)= \displaystyle\max_{j\in J}\log\Bigg(\sum_{k\in K}\prod_{l \in L}|f_{j,k,l}(z)|^{a_{j,k,l}}\Bigg)$ where $ f_{j,k,l}(z) $ are holomorphic functions and $ a_{j,k,l} \in \mathbb{R}_{>0}$ with the sets $J,K,L$ of indices $j,k,l$ being finite. Then $\varphi(z)$ is a psh function and  $e^{\varphi(z)}$ is locally Hölder continuous. By Proposition 2.5 in  \cite{KR}, we have $\varphi(z)\geq \lambda\log|z| + O(1) \text{ for } \lambda > 0 $ where $|z|= |z_1| + \cdots + |z_n|$.
\end{example}
\begin{lemma} \label{level set inc}
   Let $f$, $\varphi$, $\nu$ and $K$ be as in Theorem \ref{main2}. Then given any $N \in \mathbb{N}$  the set $\{\varphi(z) <-N\}\cap K$ is contained in the ball $B(0, Mr(N))$ where $M=\sqrt{n}e^{\frac{C}{\nu}}$ and $ r(N)=  e^{-\frac{N}{\nu}} $. 
\end{lemma}

\begin{proof}
   By Proposition \ref{toric bdd}  we have that $$  \nu \displaystyle\max_j(\log||z_j||)-C \leq \varphi(z) $$ for $z\in K$. We note that, as $\log||z||$ is an increasing function, 
   $$ \varphi(z)\geq \nu \displaystyle\max_j(\log||z_j||)-C =  \displaystyle\frac{\nu}{2}\log(\frac{1}{n}n\max_j||z_j||^2)-C  $$ on $K$ which implies that $$ \varphi(z)\geq -\frac{\nu \log n}{2}-C +  \frac{\nu}{2}\log\displaystyle\sum_{j=1}^{n} ||z_j||^2   =-\frac{\nu \log n}{2}-C+\nu\log ||z||.$$  Thus we obtain the result 
   $$ \{\varphi(z) <-N\}\cap K \subset B(0, Mr(N)).$$ 
\end{proof}
\begin{lemma}\label{mass}
    Let $u$ be an $\alpha$-Hölder continuous psh function for some $\alpha$ with $0<\alpha \leq 1$. Then we have the following estimation of the complex Monge-Amp\'ere masses of balls $$ (dd^cu)^n(B(0,r)) \leq Mr^{\alpha n}$$ for some constant $M>0$ and for all $r>0$.
\end{lemma}
\begin{proof}
    We set $B_r:=B(0,r)$ and $B:=B(0,1)$. We define the following function: $$  u_r(z)=\frac{1}{r^{\alpha}} \Big(u(rz)-u(0)\Big) $$ on $B$ for $0<r<<1$. As $u$ is $\alpha$-Hölder psh function on $B_r$, there exists $\tilde M>0$ such that 
    $$ |u(rz)-u(0)|\leq \tilde M||rz||^{\alpha}$$ which implies that 
    $$ |u_r(z)|\leq \tilde M||z||^{\alpha} $$ for $z\in B.$ 
    Thus $ ||u_r||_{L^{\infty}}\leq \tilde M.$ 
Let $\sigma:B\longrightarrow B_r$ be the biholomorphism given by $\sigma(z) = rz$. We now have 
    \begin{align*}
  (dd^cu)^n(B_r) &=  \displaystyle\int_{B_r}(dd^cu)^n = \displaystyle\int_{B}\sigma^*(dd^cu)^n \\ &= \displaystyle\int_{B}\Big(dd^c(u\circ \sigma )\Big)^n = \displaystyle\int_{B}\Big(dd^cu(rz)\Big)^n \\ 
 &= \displaystyle\int_{B}\Big(dd^c\big[r^{\alpha}u(rz)+u(0)\big]\Big)^n \\
  &= \displaystyle\int_{B}\Big(dd^c\big(r^{\alpha}u_r(z)\big)\Big)^n =  r^{\alpha n}\displaystyle\int_{B}\Big(dd^cu_r(z)\Big)^n  \leq r^{\alpha n}\tilde M^n.\end{align*}
 The last inequality follows from Chern-Levine-Nirenberg estimate (Corollary 3.4.8 in \cite{Kl}).
\end{proof}
 Now we give an estimate for the Monge-Amp\'ere measure of level sets.
\begin{lemma}\label{sublevel}
    Let $f$, $\varphi$, $\nu$ and $K$ be as in Theorem \ref{main2}. Then we have $$ \displaystyle (dd^cu)^n (\{\varphi <-N \}\cap K)  \leq M\cdot r(N)^{\alpha n}$$ where $N$ is any natural number and M is positive constant which does not depend on $N$, and $r(N)=e^{-\frac{N}{ \nu}}$. 
\end{lemma}
\begin{proof}
    Combining Lemma \ref{level set inc} with Lemma \ref{mass} , we obtain  that $$ (dd^cu)^n((\varphi<-N)\cap K)  \leq (dd^cu)^n(B(0,M'e^{\frac{-N}{\nu}}))\leq Me^{\frac{-n\alpha N}{\nu}}.$$ 
\end{proof}

Now we prove our second main result on integrability of  psh functions.

\begin{proof}[Proof of Theorem \ref{main2}]
We have
     \begin{align} \label{(3.1)}
  \displaystyle\int\limits_{K}^{}e^{-\varphi}(dd^cu)^n &= \displaystyle\sum_{N=0}^{\infty}\int\limits_{\{ -N\leq \varphi < -N+1\}\cap K}^{}e^{-\varphi}(dd^cu)^n \nonumber \\ & \leq\displaystyle\sum_{N=0}^{\infty} e^N \int\limits_{\{ -N\leq \varphi < -N+1\}\cap K}^{}(dd^cu)^n \nonumber \\ & \leq  \displaystyle\sum_{N=0}^{\infty} e^NMe^{-\frac{\alpha n N}{\nu}}\nonumber \\ & = M\displaystyle\sum_{N=0}^{\infty} e^{\Big(1-\frac{\alpha n}{ \nu}\Big)N}. 
  \end{align} 
Second inequality above follows from  Lemma \ref{sublevel} . By assumption 
  $ 1-\frac{\alpha n}{\nu} $ is negative and hence the last sum  (\ref{(3.1)}) is convergent.
  \end{proof}

  As a corollary, we improve the lower bound for $\alpha$-Hölder integrability exponent of toric psh functions. We note that Example \ref{ex5} shows that this lower bound is sharp.
\begin{corollary} Let $f$, $\varphi$, $\nu$ and $K$ be as in Theorem \ref{main2}. Then $$c_{\alpha}(\varphi,0)\geq \frac{\alpha n}{2\nu}.$$
    
\end{corollary}
\begin{proof} Take any $\beta<\alpha n$. Then $\nu\left(\frac{\beta f}{\nu}\right)=\beta<\alpha n. $ By Theorem \ref{main2} $$\displaystyle\int_Ke^{\frac{-\beta\varphi}{\nu}}(dd^c u)^n<\infty.$$ Thus $c_{\alpha}(\varphi,0)\geq \frac{\alpha n}{2\nu}.$   
\end{proof}
\section{A family of Monge-Amp\'ere Measures}
Let $\mathcal{F}_m^1(B)$ be the set of Hausdorff-Riesz measures of order $m$ as defined in \cite[Section 2]{Z} and $\text{Dir}(B,0,d\mu)$ be the Dirichlet problem where $\mu \in \mathcal{F}_m^1(B)$. If $m=2n-2 + \epsilon$ for $0<\epsilon \leq 2$, then, by \cite[Theorem 1.2]{Ch}, there exists a unique solution $u$ of the Dirichlet problem $\text{Dir}(B,0,d\mu)$ where $u $ is the $\alpha=\frac{\epsilon^2\gamma}{\epsilon +2}$-Hölder continuous psh on $\bar{B}$, for any $0<\gamma<\frac{1}{2n+1}$. Therefore, we will consider the following family of Hölder continuous psh functions

$$ \mathcal{U}_{\epsilon} =\{u \mid (dd^cu)^n=d\mu \text{ for some } \mu \in \mathcal{F}^1_{2n-2+\epsilon}  \}  $$ where $0<\epsilon \leq 2$ and $u$ is the unique solution to the Dirichlet problem $\text{Dir}(B,0,d\mu)$. For this family $\mathcal{U}_{\epsilon}$, we will look for a better integrability condition for the psh functions, bounded below by a toric psh function with isolated singularity at the origin when compared with the Theorem \ref{main2}.
\par Let $f$, $\varphi$, $\nu$ and $K$ be as in Theorem \ref{main2}, $M$ and $r(N)$ be as in Lemma \ref{level set inc},  and $u$ be in the family $\mathcal{U}_{\epsilon}$. We have the following:

\begin{align*} \label{(4.1)}
  \displaystyle\int\limits_{K}^{}e^{-\varphi}(dd^cu)^n &= \displaystyle\sum_{N=0}^{\infty}\int\limits_{\{ -N\leq \varphi < -N+1\}\cap K}^{}e^{-\varphi}(dd^cu)^n \nonumber \\ & \leq\displaystyle\sum_{N=0}^{\infty} e^N \int\limits_{\{ -N\leq \varphi < -N+1\}\cap K}^{}(dd^cu)^n \nonumber 
 \\& \leq  \displaystyle\sum_{N=0}^{\infty} e^N \mu\big( B(0,Mr(N-1) ) \cap K \big )  \nonumber  \\ & \leq  \displaystyle\sum_{N=0}^{\infty} e^N\big( Mr(N-1)\big )^{2n-2+\epsilon}  \text{ as } \mu \in \mathcal{F}^1_{2n-2 + \epsilon} \nonumber     \\ & = M'\displaystyle\sum_{N=0}^{\infty} e^{\Big(1-\frac{2n-2+\epsilon}{ \nu}\Big)N}. 
  \end{align*} 
 The second inequality above follows from Lemma \ref{level set inc}. The last sum  is convergent if and only if $\nu < 2n-2+\epsilon$. Thus we have the following Theorem: 
  \begin{theorem}\label{ex}
      Let $f$, $\varphi$, $\nu$, $K$ and $\mathcal{U}_{\epsilon}$ be as in above, and $u$ be in the family $\mathcal{U}_{\epsilon}$. If we assume that $\nu <2n-2+\epsilon$, then the integral $$   \displaystyle\int\limits_{K}^{}e^{-\varphi}(dd^cu)^n $$ is finite. 
  \end{theorem}


\noindent \textbf{Acknowledgments.} The  authors are supported by T\"UB\.ITAK 1001 Proj. No 124F144.

	\end{document}